\documentclass{article}
\usepackage[english]{babel}
\usepackage{amscd,amssymb,amsmath,amsfonts,latexsym,amsthm}
\usepackage{inputenc}
\usepackage{graphicx}
\newtheorem{pr}{Proposition}

\newtheorem{lemma}{Lemma}
\newtheorem{de}{Definition}
\newtheorem{teo}{Theorem}

\newtheorem{remark}{Remark}

\begin{document}
\title{Leibniz algebras associated with representations of the Diamond Lie algebra}

\author{Selman Uguz\footnote{Corresponding author, Department of Mathematics,
Arts and Sciences Faculty Harran University, 63120 \c{S}anliurfa,
Turkey, e-mail: selmanuguz@gmail.com, Tel: +90 414 318 3595 Fax:
+90 414 318 3541}, \ Iqbol A. Karimjanov \footnote{Department of
Algebra, University of Santiago de Compostela, 15782,
 {Spain}, e-mail: iqboli@gmail.com}, Bakhrom A. Omirov
\footnote{Bakhrom A. Omirov, Institute of Mathematics, National
University of Uzbekistan, Dormon yoli str. 29, 100125, Tashkent
(Uzbekistan), e-mail: omirovb@mail.ru}}

\maketitle
\begin{abstract} In this paper we describe some Leibniz algebras whose corresponding Lie algebra is
four-dimensional Diamond Lie algebra $\mathfrak{D}$ and the ideal generated by the squares of elements
(further denoted by $I$) is a right $\mathfrak{D}$-module. Using description \cite{Cas} of representations
of algebra $\mathfrak{D}$ in $\mathfrak{sl}(3,{\mathbb{C}})$ and $\mathfrak{sp}(4,{\mathbb{F}})$ where
${\mathbb{F}}={\mathbb{R}}$ or ${\mathbb{C}}$ we obtain the classification of above mentioned Leibniz algebras.
Moreover, Fock representation of Heisenberg Lie algebra was extended to the case of the algebra $\mathfrak{D}.$
Classification of Leibniz algebras with corresponding Lie algebra $\mathfrak{D}$ and with the ideal $I$ as a
Fock right $\mathfrak{D}$-module is presented. The linear integrable deformations in terms of the second cohomology groups of obtained finite-dimensional Leibniz algebras are described. Two computer programs in Mathematica 10 which help to calculate for a given Leibniz algebra the general form of elements of spaces $BL^2$ and $ZL^2$ are constructed, as well.
\end{abstract}

\medskip \textbf{AMS Subject Classifications (2010): 17A32, 17B30, 17B10.}

\textbf{Key words:} Diamond Lie algebra, Leibniz algebra, representation of Diamond Lie algebra, Fock representation, Heisenberg Lie algebra, linear deformation, the second group of cohomology.

\section{Introduction}

Leibniz  algebras are non-commutative analogue of Lie algebras, in the sense that adding antisymmetry to Leibniz bracket leads to coincidence of fundamental identity (Leibniz identity) with  Jacobi identity. Therefore, Lie algebra is a particular case of Leibniz algebra. Leibniz algebras were introduced by J.-L. Loday \cite{Loday} in 1993 and since then the study of Leibniz algebras has been carried on intensively. Investigation of Leibniz algebras shows that classical results on Cartan subalgebras, Levi's decomposition, Engel's and Lie's theorems, properties of solvable algebras with given nilradical and others from theory of Lie algebras have been extended to Leibniz algebras case (see \cite{Alb}, \cite{Bar1}, \cite{Bar2}, \cite{Nilrad1}, \cite{Nilrad2}, \cite{Gorbat}, \cite{Omir}). It is known that for a given Leibniz algebra the  categories of symmetric representations and anti-symmetric representations are both equivalent to the category of Lie representations over its corresponding Lie algebra \cite{LodPir}.

Recall that an algebra $L$ over a field $F$ is called a {\it Leibniz algebra} if it satisfies the following Leibniz identity: $$[x,[y,z]]=[[x,y],z]-[[x,z],y],$$ where $[ \ , \ ]$ denotes the multiplication in $L$. In fact, the ideal $I$ generated by the squares of elements of a non-Lie Leibniz algebra $L$ plays an important role since it determines the (possible) non-Lie nature of $L$. From the Leibniz identity, this ideal is contained in right annihilator of the algebra $L$.

For a Leibniz algebra $L$ we consider the natural homomorphism $\varphi$ into the quotient Lie algebra $\overline{L}=L/I$, which is called {\it corresponding Lie algebra} to Leibniz algebra $L$ (in some papers it is called a {\it liezation} of $L$).

The map $I \times \overline{L} \to I$, $(i,\overline{x}) \mapsto [i,x]$ endows $I$ with a structure of a right $\overline{L}$-module (it is well-defined due to $I$ being in a right annihilator).

Denote by $Q(L) = \overline{L} \oplus I,$ then the operation $(-,-)$ defines Leibniz algebra structure on $Q(L),$ where $$(\overline{x},\overline{y}) = \overline{[x,y]}, \quad (\overline{x},i) = [x,i],
\quad (i, \overline{x}) = 0, \quad (i,j) = 0, \qquad x, y \in L, \ i,j \in I.$$

Therefore, for a given Lie algebra $G$ and a right $G$-module $M,$ we can construct a Leibniz algebra as described above.

In \cite{Bar1} D. Barnes showed that any finite-dimensional complex Leibniz algebra is decomposed into a semidirect sum of the solvable radical and a semisimple Lie algebra (the analogue of Levi's theorem). Hence, we conclude that if the quotient algebra is isomorphic to a semisimple Lie algebra, then knowing a module over this semisimple Lie algebra one can easily obtain the description of the Leibniz algebras with corresponding semisimple Lie algebra. Therefore, it is important to study the case when the corresponding Lie algebra is solvable. Moreover, the solvability of a given Leibniz algebra is equivalent to the solvability of its corresponding  Lie algebra.

One of the approaches related to this construction is the description of such Leibniz algebras whose corresponding
Lie algebra is a given Lie algebra.

This approach was used in works \cite{Filiform}, \cite{Heisenberg} where some Leibniz algebras with corresponding Lie algebra being filiform and Heisenberg $H_n$ Lie algebras, respectively,  are described. In particular, the classification theorems were obtained for Leibniz algebras whose corresponding Lie algebras are Heisenberg and naturally graded filiform algebras and that the ideal $I$ is isomorphic to Fock module as a module over corresponding Lie algebra.

Deforming a given mathematical structure is a tool of fundamental importance in most parts of mathematics, mathematical physics and physics. Deformations and contractions have been investigated by researchers who had different approaches and goals. Tools such as cohomology, gradings, etc. which are utilized in the study of one concept, are likely to be useful for the other concept as well. The theory of deformations originated with the problem of classifying all possible pairwise non-isomorphic complex structures on a given differentiable real manifold. The concept of formal deformations of arbitrary rings and associative algebras was first investigated in 1964 by Gerstenhaber \cite{Gersten}. Later, the notion of deformation was applied to Lie algebras by Nijenhuis and Richardson \cite{Nijen}. After works of these authors the formal deformation theory was generalized in different categories. In fact, in the last fifty years, deformation theory has played an important role in algebraic geometry. The main goal is the classification of families of geometric objects when the classifying space (the so called moduli space) is a reasonable geometric space. In particular, each point of our moduli space corresponds to one geometric object (class of isomorphism). Deformation is one of the tools used to study a specific object, by deforming it into some families of "similar" structure objects. This way we get a richer picture about the original object itself. But there is also another question approached via deformation. Roughly speaking, it is the question, can we equip the set of mathematical structures under consideration (may be up to certain equivalence) with the structure of a topological or geometric space. The theory of deformations is another one of the effective approach in investigating of solvable and nilpotent Lie algebras and superalgebras. Since Leibniz algebras are generalization of Lie algebras and it is natural to apply the theory of deformations to the study of Leibniz algebras. Thanks to Balavoine \cite{Balavoine} we can apply the general principles for deformations and rigidity of Leibniz algebras. For instance, one-parameter deformations establish connection between Leibniz algebra cohomology and infinitesimal deformations \cite{Khud1}, \cite{Khud2}.

In this paper we describe Leibniz algebras with corresponding four-dimensional Diamond Lie algebra $\mathfrak{D}$
and with the ideal $I$ associated to representations of $\mathfrak{D}$ in $\mathfrak{sl}(3, {\mathbb{C}})$ and
$\mathfrak{sp}(4, {\mathbb{F}})$ with ${\mathbb{F}}={\mathbb{R}}$ and ${\mathbb{F}}={\mathbb{C}}$ (see \cite{Cas}).
 We recall the relation $[\mathfrak{D},\mathfrak{D}]=H_1$ between algebras $\mathfrak{D}$ and $H_1$. For the other
 properties of the Diamond Lie algebra and its relation with the Diamond Lie group we refer the reader
 to works \cite{Avit}, \cite{Cas}, \cite{Ludwig}, \cite{Napi} and references therein. In addition,
 we extend the notion of Fock representation of Heisenberg algebra to representation of the algebra $\mathfrak{D}$.
  Using these representations we classify Leibniz algebras with corresponding Lie algebra $\mathfrak{D}.$ Applying
  relationship between formal deformations and the second group of cohomologies we describe linear integrable
  deformations for classified finite-dimensional Leibniz algebras. In order to avoid routine calculations for
  identifying representatives of the quotient spaces $HL^2$ we developed two computer programs in Mathematica 10.

Throughout the paper omitted products in in the table of multiplication of Leibniz algebras, as well as omitted values in the expansion of 2-cocycles are assumed to be zero.

\section{Preliminaries}

In this section we give necessary definitions and preliminary results.

\begin{de} \cite{Loday} An algebra $(L,[-,-])$ over a field  $\mathbb{F}$   is called a Leibniz algebra if for any $x,y,z\in L$, the so-called Leibniz identity
\[ \big[[x,y],z\big]=\big[[x,z],y\big]+\big[x,[y,z]\big] \] holds.
\end{de}

\medskip

Further for the right annihilator of an algebra $L$ we shall use notation $Ann_r(L).$

The real Diamond Lie algebra $\mathfrak{D}_{\mathbb{R}}$ is a four-dimensional Lie algebra with bases $\{J, P_1, P_2, T\}$ and non-zero relations:
\begin{equation}[J,P_1]=P_2, \quad [J,P_2]=-P_1, \quad [P_1,P_2]=T. \end{equation}

The complexification of the Diamond Lie algebra:
$\mathfrak{D}_{\mathbb{C}}=\mathfrak{D}\otimes_\mathbb{R} \mathbb{C}$ displays the following (complex) bases:
$\{P_{+}=P_1-iP_2, \ P_{-}=P_1+iP_2, \ T, \ J\},$
where $i$ is the imaginary unit, whose nonzero commutators are
\begin{equation} \label{1.2} [J,P_{+}]=iP_{+}, \quad [J,P_{-}]=-iP_{-}, \quad [P_{+},P_{-}]=2iT.   \end{equation}

If we change base
$$J'=-iJ, \quad P_1'=P_{+}, \quad P_2'=P_{-}, \quad T'=2iT$$
then we can assume that
 \begin{equation} \label{1.3} [J,P_1]=P_1, \quad [J,P_2]=-P_2, \quad [P_1,P_2]=T.   \end{equation}

\subsection{$\bf{\mathfrak{sl}(3,\mathbb{C})}$-modules as $\mathfrak{D}_{\mathbb{C}}$-modules}
Let $\{H_1, H_2, E_1, E_2, E_{12}, F_1, F_2, F_{12}\}$ be a standard basis of $(\mathfrak{sl}_3,\mathbb{C})$ \cite{Jac} defined by $$aH_1+bH_2+cE_1+dE_2+eE_{12}+fF_1+gF_2+hF_{12}=\left(\begin{matrix}a&c&e&\\f&b-a&d\\h&g&-b\end{matrix}\right).$$

Indecomposable finite-dimensional representations of $\mathfrak{D}_{\mathbb{C}}$ by restricting those of $\mathfrak{sl}(3,\mathbb{C})$ to the two inequivalent embedding of $\mathfrak{D}_{\mathbb{C}}$ in  $\mathfrak{sl}(3,\mathbb{C})$ given the next lemma.

\begin{lemma}\label{lemma1}\cite{Cas}
The maps $\varphi:\mathfrak{D}_{\mathbb{C}}\rightarrow \mathfrak{sl}(3,\mathbb{C})$ and $\psi:\mathfrak{D}_{\mathbb{C}}\rightarrow \mathfrak{sl}(3,\mathbb{C})$ defined by
$$\varphi(P_+)=E_1, \quad \varphi(P_-)=F_{12}, \quad \varphi(J)=\frac{i}{3}(2H_1+H_2), \quad \varphi(T)=\frac{i}{2}F_2,$$
$$\psi(P_+)=E_1, \quad \psi(P_-)=E_2, \quad \psi(J)=\frac{i}{3}(H_1-H_2), \quad \psi(T)=-\frac{i}{2}E_{12},$$
are inequivalent Lie algebra embeddings.
\end{lemma}

\begin{remark} Here we changed $\psi(T)=-\frac{i}{2}F_{12}$ misprint of \cite{Cas} to the correct expression $\psi(T)=-\frac{i}{2}E_{12}.$
\end{remark}

The embeddings of Lemma \ref{lemma1} correspond to two module structures over $\mathfrak{D}_{\mathbb{C}}$ on vector space $V=\{X_1, X_2, X_3\}$:

\begin{equation}\label{sl3module1}\left\{\begin{array}{lll}
(X_1,J)=\frac{2i}{3}X_1,    & (X_2,J)=-\frac{i}{3}X_2, & (X_3,J)=-\frac{i}{3}X_3,   \\[1mm]
(X_1,P_+)=X_2,  & (X_3,P_-)=X_1, & (X_3,T)=\frac{i}{2}X_2,
\end{array}\right.
\end{equation}
\begin{equation}\label{sl3module2}\left\{\begin{array}{lll}
(X_1,J)=\frac{i}{3}X_1,    & (X_2,J)=-\frac{2i}{3}X_2, & (X_3,J)=\frac{i}{3}X_3,    \\[1mm]
(X_1,P_+)=X_2, & (X_2,P_-)=X_3, & (X_1,T)=-\frac{i}{2}X_3.
\end{array}\right.
\end{equation}

\subsection{$\bf{sp(4,\mathbb{R})}$-modules as $\mathfrak{D}_{\mathbb{R}}$-modules}
The Diamond Lie algebra $\mathfrak{D}_{\mathbb{R}}$ can be realized as a subalgebra of the simple Lie algebra $\mathfrak{sp}(4,\mathbb{R})$ through the map \cite{Cas}:
$$\theta J+\alpha P_1+ \beta P_2+\gamma T=\left(\begin{matrix}0&\alpha&\beta&2\gamma&
\\0&0&-\theta&\beta&\\0&\theta&0&-\alpha&\\0&0&0&0\end{matrix}\right)$$

Precisely, if $\{J,P_1,P_2,T \}$ is a basis of $\mathfrak{D}_{\mathbb{R}}$, then as faithful representations we take linear transformations with the matrices on the linear space $V=\{X_1, X_2, X_3,X_4\}$. We
endow the vector space $V$ with right $\mathfrak{D}_{\mathbb{R}}$-module structure as follows:
\begin{equation}\label{sp4R}\left\{\begin{array}{lll}
(X_1,P_1)=X_2,  & (X_1,P_2)=X_3, & (X_1,T)=2X_4, \\[1mm]
(X_2,J)=-X_3,   & (X_2,P_2)=X_4, & \\[1mm]
(X_3,J)=X_2,    & (X_3,P_1)=-X_4.
\end{array}\right.
\end{equation}

\subsection{$\bf{\mathfrak{sp}(4,\mathbb{C})}$-modules as $\mathfrak{D}_{\mathbb{C}}$-modules}
The Chevalley basis of $\mathfrak{sp}(4,\mathbb{C})$ (see \cite{Jac}) is defined by
$$aH_1+bH_2+cE_1+dE_2+eE_{12}+fE_{112}+gF_1+hF_2+iF_{12}+jF_{112}=
\left(\begin{matrix}a&c&e&-f\\g&b-a&d&-e&\\i&h&a-b&0&\\-j&-i&0&-a\end{matrix}\right)$$
and $\eta:\mathfrak{D}_{\mathbb{C}}\rightarrow \mathfrak{sp}(4,\mathbb{C})$ simply becomes
$$\eta(P_{+})=E_1, \quad \eta(P_{-})=F_{12}, \quad \eta(J)=i(H_1+H_2), \quad \eta(T)=\frac{i}{2}F_2 $$

\begin{remark} Here we corrected misprinted coefficient of \cite{Cas} in $\eta(T)=i F_2$ to $\frac{i}{2}$.
\end{remark}

From the above embedding we construct right $\mathfrak{D}_{\mathbb{C}}$-module $V=\{X_1, X_2, X_3, X_4\}$ in the following way:
\begin{equation}\label{sp4C}
\left\{\begin{array}{lll}
(X_1,J)=iX_1,   & (X_4,J)=-iX_4, & (X_1,P_{+})=X_2, \\[1mm]
(X_3,P_{-})=X_1,    & (X_4,P_{-})=-X_2, & (X_3,T)=\frac{i}{2} X_2.
\end{array}\right.
\end{equation}

\subsection{Fock module over Heisenberg Lie algebra }

It is known that if we denote by $\overline{{x}}$ the operator associated to position and by ${\frac{\overline{\partial}}{\partial x}}$ the one associated to momentum (acting for instance on the space $V$ of differentiable functions on a single variable), then $[\overline{x},{\frac{\overline{\partial}}{\partial x}}]=\overline{1}_V$. Thus we can identify the subalgebra generated by $\overline{1},\overline{{x}}$ and
${\frac{\overline{\partial}}{\partial x}}$ with the three-dimensional Heisenberg Lie algebra $H_1$ whose multiplication table in the basis $\{\overline{1},\overline{x},\frac{\overline{\partial}}{\partial x}\}$ has a unique non-zero product $[\overline{x},\frac{\overline{\partial}}{\partial x}]=\overline{1}$.

For a given Heisenberg algebra $H_1$ this explanation gives rise to the so-called {\it Fock module} over $H_1$,  the linear space ${\mathbb{F}}[x]$ of polynomials on $x$ ($\mathbb{F}$ denotes
the algebraically closed field with zero characteristic) with the action induced by
\begin{equation}\label{Fock}
\begin{array}{lll} ( p(x),\overline{1})& \mapsto &
p(x)\\{} ( p(x),\overline{x})& \mapsto & xp(x)\\{}
(p(x),\frac{\overline{\partial}}{\partial x})& \mapsto & \frac
{\partial}{\partial x}(p(x))
\end{array}
\end{equation}
for any $p(x) \in \mathbb{F}[x].$

\subsection{Linear deformations of Leibniz algebras}

We call a vector space $M$ a module over a Leibniz algebra $L$ if there are two bilinear maps:
$[-,-]:L\times M \rightarrow M$ and $[-,-]:M\times L \rightarrow M$
satisfying the following three axioms
$$\begin{array}{ll}
[m,[x,y]] =[[m,x],y]-[[m,y],x],\\[1mm]
[x,[m,y]] =[[x,m],y]-[[x,y],m],\\[1mm]
[x,[y,m]] =[[x,y],m]-[[x,m],y],\\[1mm]
\end{array}
$$
for any $m\in M$, $x, y \in L$.

Given a Leibniz algebra $L$, let $C^n(L,M)$ be the space of all $F$-linear homogeneous mappings $L^{\otimes n} \rightarrow M$, $n \geq 0$ and $C^0(L,M) = M$.

Let $d^n : C^n(L,M) \rightarrow C^{n+1}(L,M)$ be an $F$-homomorphism
defined by
 \begin{multline*}
(d^nf)(x_1, \dots , x_{n+1}): = [x_1,f(x_2,\dots,x_{n+1})]
+\sum\limits_{i=2}^{n+1}(-1)^{i}[f(x_1,
\dots, \widehat{x}_i, \dots , x_{n+1}),x_i]\\
+\sum\limits_{1\leq i<j\leq {n+1}}(-1)^{j+1}f(x_1, \dots,
x_{i-1},[x_i,x_j], x_{i+1}, \dots , \widehat{x}_j, \dots
,x_{n+1}),
\end{multline*}
where $f\in C^n(L,M)$ and $x_i\in L$. Since the derivative operator $d=\sum\limits_{i \geq 0}d^i$ satisfies the property $d\circ d = 0$, the $n$-th cohomology group is well defined and $$HL^n(L,M) = ZL^n(L,M)/ BL^n(L,M),$$
where the elements $ZL^n(L,M)$ and $BL^n(L,M)$) are called {\it $n$-cocycles} and {\it $n$-coboundaries}, respectively.

The elements $f\in BL^2(L,L)$ and $\varphi \in ZL^2(L,L)$ are defined as follows \begin{equation}\label{E.B2} f(x,y) = [d(x),y] + [x,d(y)] - d([x,y]) \ \mbox{for some linear map} \ d\end{equation} and
\begin{equation}\label{E.Z2}(d^2\varphi)(x,y,z)=[x,\varphi(y,z)]
 - [\varphi(x,y), z] + [\varphi(x,z), y] + \varphi(x, [y,z]) - \varphi([x,y],z) + \varphi([x,z],y)=0. \end{equation}

{\it A formal deformation of a Leibniz algebra} $L$ is a one-parameter family $L_t$ of Leibniz algebras with the bracket $$\mu_t = \mu_0 + t\varphi_1 + t^2\varphi_2 + \cdots,$$ where $\varphi_i$ are 2-cochains, i.e., elements of $Hom(L \otimes L, L)= C^2(L, L)$.

Two deformations $L_t, \ L'_t$ with corresponding laws $\mu_t, \ \mu'_t$ are {\it equivalent} if there exists a linear automorphism $f_t = id + f_1 t + f_2 t^2 + \cdots$ of $L$, where $f_i$ are elements of $C^1(L, L)$ such that the following equation holds
$$\mu'_t(x, y) = f_t^{-1}(\mu_t(f_t(x), f_t(y))) \ \ \text{for} \  x, y \in L.$$

The Leibniz identity for the algebras $L_t$ implies that the 2-cochain $\varphi_1$ should satisfy the equality $d^2\varphi_1 = 0,$ i.e. $\varphi_1\in ZL^2(L,L)$. If $\varphi_1$ vanishes identically, then the first non vanishing $\varphi_i$ is 2-cocycle.

If $\mu'_t$ is an equivalent deformation with cochains $\varphi_i'$, then $\varphi_1' -\varphi_1 = d^1f_1$, hence every equivalence class of deformations defines uniquely an element of $HL^2(L, L)$.

It should be noted that the condition that linear deformation (that is, $\mu_t = \mu_0 + t\varphi_1$) is a Leibniz algebra (we say $\mu_t$ is {\it integrable}) implies two restrictions on 2-cochain $\varphi_1$:
the fist one is $\varphi_1\in ZL^2(L,L)$ and the second one is
\begin{equation}\label{Leibnizcocycle}\varphi_1(x, \varphi_1(y,z)) - \varphi_1(\varphi_1(x,y),z) + \varphi_1(\varphi_1(x,z),y)=0. \end{equation}

\section{Main result}

In this section we present descriptions of Leibniz algebras with corresponding the Diamond Lie algebra $\mathfrak{D}$ and by identifying ideal $I$ with $\mathfrak{D}$-modules discussed in previous section.
 Using computer program in Mathematica 10, we describe linear integrable deformations for obtained finite-dimensional Leibniz algebras.

\subsection{Leibniz algebras with the ideal $I$ as $\mathfrak{D}_{\mathbb{C}}$-modules by restriction of $\mathfrak{sl}(3,\mathbb{C})$.}

In this subsection we are going to describe Leibniz algebras $L$ such that $L/I \cong \mathfrak{D}_{\mathbb{C}}$ and the ideal $I$ is identified as a right $\mathfrak{D}_{\mathbb{C}}$-modules by restriction of $\mathfrak{sl}(3,\mathbb{C})$.

For the shortness, instead Leibniz identity $[X,[Y,Z]]=[[X,Y],Z]-[[X,Z],Y]$ we will use below the notation $\{X,Y,Z\}$.

\begin{lemma}\label{12} Let $L$ be a Leibniz algebra such that $L/I\cong\mathfrak{\overline{D}}_{\mathbb{C}}$, where $\mathfrak{\overline{D}}_{\mathbb{C}}$ is the Diamond Lie algebra and $I$ is its right $\mathfrak{\overline{D}}_{\mathbb{C}}$-module.
If there  exists a basis $\{X_1,X_2,\dots,X_n\}$ of $I$ such that $[X_i,J]=\alpha_iX_i, \ \alpha_i {\notin}
\{-2,0,2\}$  for $1\leq i\leq n,$ then
$$[\mathfrak{D},\mathfrak{D}]\subseteq\mathfrak{D}.$$
\end{lemma}
\begin{proof} Here we shall use the table of multiplication (\ref{1.3}) of the complex Diamond Lie algebra.
Let us assume that $[J,J]=\sum\limits_{i=1}^nm_iX_i.$ Then by setting $J':=J-\sum\limits_{i=1}^n\frac{m_i}{\alpha_i}X_i,$ we can assume that $[J,J]=0$.

Let us denote
$$[J,P_1]=P_1+\sum\limits_{i=1}^nq_iX_i, \quad [J,P_2]=-P_2+\sum\limits_{i=1}^nr_iX_i.$$

Taking the following basis transformation:
$$ J'=J, \quad P_1'=P_1+\sum\limits_{i=1}^nq_iX_i, \quad P_2'=P_2-\sum\limits_{i=1}^nr_iX_i, \quad T'=[P_1',P_2'], $$
we can assume that
$$[J,P_1]=P_1, \quad [J,P_2]=-P_2, \quad [P_1,P_2]=T.$$

Applying the  Leibniz identity for the triples $\{J,J,P_1\}, \ \{J,J,P_2\}$ we derive
$$[P_1,J]=-[J,P_1], \quad [P_2,J]=-[J,P_2].$$

We put $$[J,T]=\sum_{i=1}^{n}t_iX_i.$$

Considering the Leibniz identity for the triple $\{J,J,T\}$ and taking into account the condition $\alpha_i\neq0,$ we get $[J,T]=0.$

Similarly, from the  Leibniz identity for the following triples we obtain:

$$\left\{\begin{array}{lll}
\{J,P_1,P_2\}, & \Rightarrow & [P_2,P_1]=-[P_1,P_2], \\[1mm]
\{J,P_1,T\}, & \Rightarrow & [P_1,T]=0, \\[1mm]
\{J,P_2,T\}, & \Rightarrow & [P_2,T]=0, \\[1mm]
\{P_1,P_2,T\}, & \Rightarrow & [T,T]=0,  \\[1mm]
\{P_1,J,P_2\}, & \Rightarrow & [T,J]=0. \\[1mm]
\end{array}\right.$$

Taking into account the condition of proposition $\alpha_i\notin\{-2,0,2\}$ in the Leibniz identity for the triples:
 $$ \{P_1,J,P_1\}, \ \{P_2,J,P_2\}, \ \{P_1,P_1,P_2\}, \ \{P_2,P_2,P_1\}$$  we obtain
  the products $$[P_1,P_1]=[P_2,P_2]=[T,P_1]=[T,P_2]=0,$$ which complete the proof of the lemma.
\end{proof}

In the following theorems we consider the case when the ideal $I$ of the algebra $L$ is defined by right $\mathfrak{\overline{D}}_{\mathbb{C}}$-modules (\ref{sl3module1}) and (\ref{sl3module2}), respectively.

\begin{teo} An arbitrary Leibniz algebra with corresponding Lie algebra $\mathfrak{\overline{D}}_{\mathbb{C}}$ and $I$ associated with $\mathfrak{\overline{D}}_{\mathbb{C}}$-module defined by (\ref{sl3module1}) admits a basis $\{ J,P_+,P_-,T,X_1, X_2,X_3\}$ such that the table of multiplication of an algebra has the following form:
$$L_1: \quad \left\{\begin{array}{lll}
[J,P_+]=iP_+,   & [J,P_-]=-iP_-, & [P_+,P_-]=2iT,   \\[1mm]
[P_+,J]=-iP_+,  & [P_-,J]=iP_-, & [P_-,P_+]=-2iT, \\[1mm]
[X_1,J]=\frac{2}{3}iX_1,    & [X_2,J]=-\frac{1}{3}iX_2, & [X_3,J]=-\frac{1}{3}iX_3,   \\[1mm]
[X_1,P_+]=X_2,  & [X_3,P_-]=X_1, & [X_3,T]=\frac{i}{2}X_2.
\end{array}\right.$$
\end{teo}
\begin{proof} The proof is following from Lemma \ref{12}.
\end{proof}

\begin{teo} An arbitrary Leibniz algebra with corresponding Lie algebra $\mathfrak{\overline{D}}_{\mathbb{C}}$ and $I$ associated with $\mathfrak{\overline{D}}_{\mathbb{C}}$-module defined by (\ref{sl3module2}) admits a basis $\{ J,P_+,P_-,T,X_1, X_2,X_3\}$ of $L_1$ such that the table of multiplication of an algebra has the following form:
$$L_2: \quad \left\{\begin{array}{lll}
[J,P_+]=iP_+,   & [J,P_-]=-iP_-, & [P_+,P_-]=2iT,   \\[1mm]
[P_+,J]=-iP_+,  & [P_-,J]=iP_-, & [P_-,P_+]=-2iT, \\[1mm]
[X_1,J]=\frac{1}{3}iX_1,    & [X_2,J]=-\frac{2}{3}iX_2, & [X_3,J]=\frac{1}{3}iX_3,    \\[1mm]
[X_1,P_+]=X_2, & [X_2,P_-]=X_3, & [X_1,T]=-\frac{i}{2}X_3.
\end{array}\right.$$
\end{teo}
\begin{proof}The proof is following from Lemma \ref{12}.\end{proof}

\subsection{Leibniz algebras with the ideal $I$ as $\overline{\mathfrak{D}}_{\mathbb{R}}$-modules by restriction of  $\mathfrak{sp}(4,\mathbb{R})$.}

In this subsection we shall describe real Leibniz algebras $L$ such that $L/ I \cong  \overline{\mathfrak{D}}_{\mathbb{R}}$ and the ideal $I$ is a right faithful representation $\overline{\mathfrak{D}}_{\mathbb{R}}$ in $sp(4, \mathbb{R})$, defined by right module (\ref{sp4R}). Let $\{ J,P_1,P_2,T,X_1, X_2,X_3,X_4\}$ be a basis of $L.$

We set
\begin{equation}\label{eq11}\left\{\begin{array}{ll}
[J,P_1]=P_2+\sum\limits_{i=1}^4a_iX_i,  &[P_1,J]=-P_2+\sum\limits_{i=1}^4d_iX_i, \\[1mm]
[J,P_2]=-P_1+\sum\limits_{i=1}^4b_iX_i, &[P_2,J]=P_1+\sum\limits_{i=1}^4k_iX_i, \\[1mm]
[P_1,P_2]=T+\sum\limits_{i=1}^4c_iX_i,  & [P_2,P_1]=-T+\sum\limits_{i=1}^4l_iX_i, \\[1mm]
[J,J]=\sum\limits_{i=1}^4m_iX_i, & [P_1,P_1]=\sum\limits_{i=1}^4n_iX_i, \\[1mm]
[P_2,P_2]=\sum\limits_{i=1}^4p_iX_i, & [T,T]=\sum\limits_{i=1}^4q_iX_i, \\[1mm]
[J,T]=\sum\limits_{i=1}^4r_iX_i, & [T,J]=\sum\limits_{i=1}^4s_iX_i, \\[1mm]
[P_1,T]=\sum\limits_{i=1}^4t_iX_i, & [T,P_1]=\sum\limits_{i=1}^4u_iX_i, \\[1mm]
[P_2,T]=\sum\limits_{i=1}^4v_iX_i, & [T,P_2]=\sum\limits_{i=1}^4w_iX_i.
\end{array}\right.\end{equation}

In the following lemma we describe the table of multiplications of the Leibniz algebras under above conditions.
\begin{lemma}\label{lemma2} There exists a basis $\{ J,P_1,P_2,T,X_1, X_2,X_3,X_4\}$ of $L$ such that the table of multiplication has the following form:
$$L(\alpha_1,\alpha_2)=\left\{\begin{array}{ll}
[J,P_1]=P_2,    &[P_1,J]=-P_2, \\[1mm]
[J,P_2]=-P_1,   &[P_2,J]=P_1, \\[1mm]
[P_1,P_2]=T,    & [P_2,P_1]=-T, \\[1mm]
[J,J]=\alpha_1X_4, & [P_1,P_1]=\alpha_2X_1, \\[1mm]
[P_2,P_2]=\alpha_2X_1, & [J,T]=2\alpha_2X_1,\\[1mm]
[P_1,T]=-2\alpha_2X_3, & [T,P_1]=3\alpha_2X_3, \\[1mm]
[P_2,T]=2\alpha_2X_2, & [T,P_2]=-3\alpha_2X_2, \\[1mm]
[X_1,P_1]=X_2,  & [X_1,P_2]=X_3,   \\[1mm]
[X_1,T]=2X_4,   & [X_2,J]=-X_3,       \\[1mm]
[X_2,P_2]=X_4,  & [X_3,J]=X_2, \\[1mm]
[X_3,P_1]=-X_4, &
\end{array}\right.$$
where parameters $\alpha_1, \ \alpha_2 \in \mathbb{C}$.
\end{lemma}
\begin{proof} Let Leibniz algebra $L$ has the products \eqref{eq11}. Then taking change of basis elements as follows $$\begin{array}{ll}
J'=J+m_3X_2-m_2X_3, & P_1' = P_1-\sum\limits_{i=1}^3b_iX_i-(b_4+m_3)X_4,\\[3mm]
P_2' = P_2+\sum\limits_{i=1}^3a_iX_i+(a_4+m_2)X_4, & T' = T+\sum\limits_{i=1}^2c_iX_i+(c_3-b_1)X_3+(c_4-b_2)X_4,
\end{array}$$ we can assume that
$$[J,P_1]=P_2,\quad [J,P_2]=-P_1, \quad [P_1,P_2]=T, \quad [J,J]=m_1X_1+m_4X_4.$$

From the Leibniz identity for the triples below we have
$$\left\{\begin{array}{lll}
\{J,T,J\}, & \Rightarrow & r_2=r_3=m_1=0, \\[1mm]
\{J,P_1,J\}, & \Rightarrow & [P_2,J]=P_1, \\[1mm]
\{J,P_2,J\}, & \Rightarrow & [P_1,J]=-P_2, \\[1mm]
\{J,P_1,T\}, & \Rightarrow & v_1=v_3=v_4=0, v_2=r_1, \\[1mm]
\{J,P_2,T\}, & \Rightarrow & t_1=t_2=t_4=0, t_3=-r_1, \\[1mm]
\{P_1,T,P_2\}, & \Rightarrow & q_1=q_2=q_3=q_4=0, \\[1mm]
\{P_1,P_1,T\}, & \Rightarrow & n_1=\frac{1}{2}r_1, \\[1mm]
\{P_2,P_2,T\}, & \Rightarrow & p_1=\frac{1}{2}r_1,  \\[1mm]
\{J,P_1,P_2\}, & \Rightarrow & p_2=-n_2, p_3=-n_3, p_4=r_4-n_4, \\[1mm]
\{P_1,J,P_2\}, & \Rightarrow & s_1=0, s_2=2n_2, s_3=2n_3, s_4=2n_4-r_4, \\[1mm]
\{P_1,P_1,P_2\}, & \Rightarrow & u_1=u_2=0, u_3=\frac{3}{2}r_1, u_4=n_2, \\[1mm]
\{P_2,J,P_1\}, & \Rightarrow & n_2=\frac{1}{4}l_3, n_3=-\frac{1}{4}l_2, n_4=\frac{1}{2}r_4, \\[1mm]
\{P_1,J,P_1\}, & \Rightarrow & l_1=l_2=l_3=l_4=0, \\[1mm]
\{P_2,P_1,P_2\}, & \Rightarrow & w_1=w_3=w_4=0, w_2=-\frac{3}{2}r_1.
\end{array}\right.$$

Summarizing all obtained restrictions and setting $\alpha_1:=m_4, \alpha_2:=\frac{1}{2}r_1, \alpha_3:=\frac{1}{2}r_4,$ we have the family of algebras
$$\left\{\begin{array}{ll}
[J,P_1]=P_2,    &[P_1,J]=-P_2, \\[1mm]
[J,P_2]=-P_1,   &[P_2,J]=P_1, \\[1mm]
[P_1,P_2]=T,    & [P_2,P_1]=-T, \\[1mm]
[J,J]=\alpha_1X_4, & [P_1,P_1]=\alpha_2X_1+\alpha_3X_4, \\[1mm]
[P_2,P_2]=\alpha_2X_1+\alpha_3X_4, & [J,T]=2\alpha_2X_1+2\alpha_3X_4,\\[1mm]
[P_1,T]=-2\alpha_2X_3, & [T,P_1]=3\alpha_2X_3, \\[1mm]
[P_2,T]=2\alpha_2X_2, & [T,P_2]=-3\alpha_2X_2, \\[1mm]
[X_1,P_1]=X_2,  & [X_1,P_2]=X_3,   \\[1mm]
[X_1,T]=2X_4,   & [X_2,J]=-X_3,       \\[1mm]
[X_2,P_2]=X_4,  & [X_3,J]=X_2, \\[1mm]
[X_3,P_1]=-X_4. &
\end{array}\right.$$

Finally making the change of basis elements
$$J'=J-\alpha_3X_1, \quad P_1'=P_1+\alpha_3X_3, \quad P_2'=P_2-\alpha_3X_2,$$
we get the family $L(\alpha_1,\alpha_2).$
\end{proof}

\begin{teo}\label{theorem5} An arbitrary Leibniz algebra of the family $L(\alpha_1,\alpha_2)$ is isomorphic to one of the following pairwise non-isomorphic algebras:
$$L(1,0), \quad L(1,1), \quad L(-1,1), \quad L(0,0), \quad L(0,1).$$
\end{teo}
\begin{proof} In order to achieve our goal we shall consider isomorphism (basis transformation) inside the family
$L(\alpha_1,\alpha_2).$ Note that element $J, P_1, X_1$ generate the algebra. Therefore, we take the general transformation of these basis elements:
$$\begin{array}{ll}
J'=A_1J+A_2P_1+A_3P_2+A_4T+\sum\limits_{i=1}^4A_{i+4}X_i,\\[3mm]
P_1'=B_1J+B_2P_1+B_3P_2+B_4T+\sum\limits_{i=1}^4B_{i+4}X_i,\\[3mm]
X_1'=\sum\limits_{i=1}^4C_{i}X_i,
\end{array}$$
with $(A_1B_2-A_2B_1)C_1\neq0$.

Let us generate the rest basis elements $P_2', T', X_2', X_3', X_4':$

$P_2'=[J',P_1']=(A_3 B_1 - A_1
B_3) P_1 + (-A_2 B_1 + A_1 B_2) P_2 + (-A_3 B_2 + A_2 B_3) T +
(A_2 B_2 + A_3 B_3 + 2 A_1 B_4) \alpha_2 X_1 + (A_7 B_1 + A_5 B_2
+ (-3 A_4 B_3 + 2 A_3 B_4) \alpha_2) X_2 + (-A_6 B_1 + A_5 B_3 +
(3 A_4 B_2 - 2 A_2 B_4) \alpha_2) X_3 + (-A_7 B_2 + A_6 B_3 + 2
A_5 B_4 + A_1 B_1 \alpha_1) X_4,$\\

$T'=[P_1',P_2']=B_1 (A_2 B_1 - A_1 B_2) P_1 +  B_1 (A_3 B_1 - A_1
B_3) P_2 + (-A_2 B_1 B_2 + A_1 B_2^2 - A_3 B_1 B_3+A_1 B_3^2) T
-B_1 (A_3 B_2 - A_2 B_3) \alpha_2 X_1 + (A_3B_1 B_5 - A_1 B_3 B_5
- 2 A_3 B_2 B_3 \alpha_2 + 2 A_2 B_3^2 \alpha_2 + 3 A_2 B_1 B_4
\alpha_2 - 3 A_1 B_2 B_4 \alpha_2) X_2 + (-A_2 B_1 B_5 + A_1 B_2
B_5 + 2 A_3 B_2^2 \alpha_2 - 2 A_2 B_2 B_3 \alpha_2 + 3 A_3 B_1
B_4 \alpha_2 -3 A_1 B_3 B_4 \alpha_2) X_3 + (-2 A_3 B_2 B_5 + 2
A_2 B_3 B_5 - A_2 B_1 B_6 + A_1 B_2 B_6 - A_3 B_1 B_7 + A_1 B_3
B_7 ) X_4,$\\

$X_2'=[X_1',P_1']=(B_2 C_1 + B_1 C_3) X_2 + (B_3 C_1 - B_1 C_2) X_3 + (2 B_4 C_1 + B_3 C_2 - B_2 C_3) X_4,$\\

$X_3'=[X_1',P_2']=(A_3 B_1 - A_1 B_3) C_1 X_2 - (A_2 B_1 - A_1 B_2) C_1 X_3 + (-2 A_3 B_2 C_1 + 2 A_2 B_3 C_1 - A_2 B_1 C_2 + A_1 B_2 C_2 - A_3 B_1 C_3 + A_1 B_3 C_3) X_4,$\\

$X_4'=[X_2',P_2']=((A_3 B_1 - A_1 B_3) (-B_3 C_1 + B_1 C_2) + (-A_2 B_1 + A_1 B_2) (B_2 C_1 + B_1 C_3)) X_4.$

\

Let us consider

$[T',P_1']=B_1^2 (A_3 B_1 - A_1 B_3) P_1 -
 B_1^2 (A_2 B_1 - A_1 B_2) P_2 + B_1^2 (-A_3 B_2 + A_2 B_3) T +
 B_1 (A_2 B_1 B_2 - A_1 B_2^2 + A_3 B_1 B_3 - A_1 B_3^2) \alpha_2 X_1 + (-A_2 B_1^2 B_5 + A_1 B_1 B_2 B_5 +
    A_3 B_1 B_2^2 \alpha_2 + 2 A_2 B_1 B_2 B_3 \alpha_2-3 A_1 B_2^2 B_3 \alpha_2 + 3 A_3 B_1 B_3^2 \alpha_2 -
    3 A_1 B_3^3 \alpha_2 + 5 A_3 B_1^2 B_4 \alpha_2 - 5 A_1 B_1 B_3 B_4 \alpha_2) X_2 + (-A_3 B_1^2 B_5 + A_1 B_1 B_3 B_5 -    3 A_2 B_1 B_2^2 \alpha_2 + 3 A_1 B_2^3 \alpha_2 - 2 A_3 B_1 B_2 B_3 \alpha_2 - A_2 B_1 B_3^2 \alpha_2 +
    3 A_1 B_2 B_3^2 \alpha_2 - 5 A_2 B_1^2 B_4 \alpha_2 + 5 A_1 B_1 B_2 B_4 \alpha_2) X_3 + (A_2 B_1 B_2 B_5 - A_1 B_2^2 B_5+
    A_3 B_1 B_3 B_5 - A_1 B_3^2 B_5 - 2 A_3 B_2^3 \alpha_2 + 2 A_2 B_2^2 B_3 \alpha_2 - 2 A_3 B_2 B_3^2 \alpha_2 +
    2 A_2 B_3^3 \alpha_2 - 5 A_3 B_1 B_2 B_4 \alpha_2 +  5 A_2 B_1 B_3 B_4 \alpha_2) X_4.$

\

On the other hand, we have

$[T',P_1']=3\alpha_2'X_3'=3\alpha_2'((A_3 B_1 - A_1 B_3) C_1
X_2 - (A_2 B_1 - A_1 B_2) C_1 X_3 + (-2 A_3 B_2 C_1 + 2 A_2 B_3
C_1 - A_2 B_1 C_2 +  A_1 B_2 C_2 - A_3 B_1 C_3 + A_1 B_3 C_3)
X_4).$

Comparing the coefficients at the appropriate basis elements, we get the restrictions:
$$\left\{\begin{array}{lll}
B_1^2 (A_3 B_1 - A_1 B_3)=0, \\[1mm]
B_1^2 (A_2 B_1 - A_1 B_2)=0, \\[1mm]
B_1^2 (-A_3 B_2 + A_2 B_3)=0, \\[1mm]
B_1 (A_2 B_1 B_2 - A_1 B_2^2 + A_3 B_1 B_3 - A_1 B_3^2) \alpha_2=0.
\end{array}\right.$$

Let us assume that $B_1\neq0$. Then the above restrictions transform to the following
$$\left\{\begin{array}{lll}
A_3 B_1 - A_1 B_3=0, \\[1mm]
A_2 B_1 - A_1 B_2=0, \\[1mm]
-A_3 B_2 + A_2 B_3=0, \\[1mm]
(A_2 B_1 B_2 - A_1 B_2^2 + A_3 B_1 B_3 - A_1 B_3^2) \alpha_2=0
\end{array}\right.$$
and

$P_2'=(A_2 B_2+ A_3 B_3 + 2 A_1 B_4) \alpha_2 X_1 +(A_7 B_1 + A_5
B_2 - 3 A_4 B_3 \alpha_2 + 2 A_3 B_4 \alpha_2) X_2 + (-A_6B_1
+ A_5 B_3 + 3 A_4 B_2 \alpha_2 - 2 A_2 B_4 \alpha_2) X_3+ (-A_7
B_2 + A_6 B_3 + 2 A_5 B_4 + A_1 B_1 \alpha_1) X_4.$

This leads $T'=[P_1',P_2']=0$,  which is a contradiction with assumption $B_1\neq0$.

Hence $B_1=0.$

Consider the product

$[T',P_1']=(-3 A_1 B_2^2 B_3 \alpha_2 -3 A_1 B_3^3 \alpha_2 ) X_2 + (3 A_1 B_2^3 \alpha_2 + 3 A_1 B_2 B_3^2 \alpha_2) X_3 + (-A_1B_2^2 B_5 - A_1 B_3^2 B_5 - 2 A_3 B_2^3 \alpha_2 + 2 A_2 B_2^2 B_3
\alpha_2 - 2 A_3 B_2 B_3^2 \alpha_2 + 2 A_2 B_3^3 \alpha_2) X_4.$

On the other hand,

$[T',P_1']=3\alpha_2'X_3' =- A_1 B_3 C_13\alpha_2' X_2 + A_1 B_2 C_1 3\alpha_2'X_3+(-2 A_3 B_2 C_1 + 2 A_2 B_3 C_1 +  A_1 B_2 C_2 + A_1 B_3 C_3)3\alpha_2' X_4.$

Therefore, we obtain
$$\alpha_2'=\frac{(B_2^2+B_3^2) \alpha_2}{C_1}.$$

Similarly, by considering the products $[P_2',J'], \ [J',J']$ we derive $A_1=1$ and
$$\alpha_1'=\frac{\alpha_1}{(B_2^2+B_3^2)C_1}.$$

Note that determinant of the basis transformation is equal to $(B_2^2+B_3^2)^4 C_1^4$, consequently $(B_2^2+B_3^2)C_1\neq0$.

{\bf Case 1.}  Let $\alpha_2\neq0.$ Then by putting $C_1=(B_2^2+B_3^2) \alpha_2$ we get $\alpha_2'=1$ and
$\alpha_1'=\frac{\alpha_1}{\alpha_2(B_2^2+B_3^2)^2}.$

If $\alpha_1=0$, then we obtain the algebra $L(1,0)$.

If $\alpha_1\neq 0,$ then taking by $B_2, B_3\in \mathbb{R}$ as a solution of the equation $B_2^2+B_3^2=\sqrt{\mid \frac{\alpha_1}{\alpha_2} \mid}$ we get the algebras $L(1,1)$ and $L(-1,1).$

{\bf Case 2.} Let $\alpha_2=0.$ Then $\alpha_2'=0.$

If $\alpha_1=0,$ then we obtain the algebra $L(0,0).$

If $\alpha_1\neq0,$ then taking $C_1=\frac{\alpha_1}{B_2^2 + B_3^2}$ we get the algebra $L(1,0).$
\end{proof}

\subsection{Leibniz algebras with the ideal $I$ as $\overline{\mathfrak{D}}_{\mathbb{C}}$-modules by restriction of  $\mathfrak{sp}(4,\mathbb{C})$.}

The main result of this subsection describes Leibniz algebras with corresponding the complex Diamond Lie algebra and the ideal $I$ corresponding to right module over the algebra $\overline{\mathfrak{D}}_{\mathbb{C}}$ by considering it as subalgebra of $\mathfrak{sp}(4,\mathbb{C}).$ In this case we have eight-dimensional Leibniz algebra $M$ with a basis $\{ J, P_{+}, P_{-}, T, X_1, X_2, X_3, X_4\}$ and with the products (\ref{sp4C}).

Let us introduce denotations:
\begin{equation}\label{eq7}\left\{\begin{array}{ll}
[J,P_{+}]=iP_{+}+\sum\limits_{i=1}^4a_iX_i, &[P_{+},J]=-iP_{+}+\sum\limits_{i=1}^4d_iX_i, \\[1mm]
[J,P_{-}]=-iP_{-}+\sum\limits_{i=1}^4b_iX_i,    &[P_{-},J]=iP_{-}+\sum\limits_{i=1}^4k_iX_i, \\[1mm]
[P_{+},P_{-}]=2iT+\sum\limits_{i=1}^4c_iX_i,    & [P_{-},P_{+}]=-2iT+\sum\limits_{i=1}^4l_iX_i, \\[1mm]
[J,J]=\sum\limits_{i=1}^4m_iX_i, & [P_{+},P_{+}]=\sum\limits_{i=1}^4n_iX_i, \\[1mm]
[P_{-},P_{-}]=\sum\limits_{i=1}^4p_iX_i, & [T,T]=\sum\limits_{i=1}^4q_iX_i, \\[1mm]
[J,T]=\sum\limits_{i=1}^4r_iX_i, & [T,J]=\sum\limits_{i=1}^4s_iX_i, \\[1mm]
[P_{+},T]=\sum\limits_{i=1}^4t_iX_i, & [T,P_{+}]=\sum\limits_{i=1}^4u_iX_i, \\[1mm]
[P_{-},T]=\sum\limits_{i=1}^4v_iX_i, & [T,P_{-}]=\sum\limits_{i=1}^4w_iX_i.
\end{array}\right.\end{equation}

In the following lemma we present the description of multiplications table of Leibniz algebra under the restrictions of this subsection.

\begin{lemma}\label{lemma2} An arbitrary Leibniz algebra satisfying the above conditions admits a basis $\{ J,P_{+},P_{-},T,X_1, X_2,X_3,X_4\}$ such that the table of multiplications of the algebra in this  basis has the following form:
\begin{equation}\label{eq1}M(\alpha)=\left\{\begin{array}{ll}
[J,P_{+}]=iP_{+},   &[P_{+},J]=-iP_{+}, \\[1mm]
[J,P_{-}]=-iP_{-},  &[P_{-},J]=iP_{-}, \\[1mm]
[P_{+},P_{-}]=2iT,  & [P_{-},P_{+}]=-2iT, \\[1mm]
[J,J]=\alpha X_2, &  [X_1,J]=iX_1, \\[1mm]
[X_4,J]=-iX_4, & [X_1,P_{+}]=X_2, \\[1mm]
[X_3,P_{-}]=X_1, &  [X_4,P_{-}]=-X_2,\\[1mm]
[X_3,T]=\frac{i}{2} X_2.
\end{array}\right.\end{equation}
\end{lemma}

\begin{proof}

Let us take the change of basis elements:
$$\begin{array}{l}
J'=J+im_1X_1+2 i r_2 X_3-im_4X_4, \\[3mm]
P_{+}' = P_{+}-ia_1X_1+(-ia_2+m_1)X_2-ia_3X_3-ia_4X_4,\\[3mm]
P_{-}' = P_{-}+(ib_1-2r_2)X_1+(ib_2-m_4)X_2+ib_3X_3+ib_4X_4, \\[3mm]
T' = T-\frac{1}{2}(a_3+ic_1)X_1+\frac{1}{2}(a_4-ic_2)X_2-\frac{1}{2}ic_3X_3-\frac{1}{2}ic_4X_4.
\end{array}.$$
Then using the products \eqref{eq7}, we can assume that
$$[J,P_{+}]=iP_{+},\quad [J,P_{-}]=-iP_{-}, \quad [P_{+},P_{-}]=2iT,$$
$$[J,J]=m_2X_2+m_3X_3, \quad [J,T]=r_1X_1+r_3X_3+r_4X_4.$$

From the following chain of equalities
$$0=[J,[T,J]]=[[J,T],J]-[[J,J],T]=[r_1X_1+r_3X_3+r_4X_4,J]-[m_2X_2+m_3X_3,T]=$$
$$=ir_1X_1-ir_4X_4-\frac{1}{2}im_3X_2$$
we get $r_1=r_4=m_3=0,$ i.e., $[J,T]=r_3X_3$ and $[J,J]=m_2X_2.$
Similarly, applying the Leibniz identity for the triples $\{J, P_{+}, J\}$ and $\{J, P_{-}, J\}$ we obtain $$[P_{-},J]=iP_{-}, \quad [P_{+},J]=-iP_{+}.$$

Consider the Leibniz identity for the following triples of elements:
$$\left\{\begin{array}{lll}
\{J,P_{+},T\}, & \Rightarrow & t_1=t_2=t_3=t_4=0, \\[1mm]
\{P_{+},J,P_{+}\}, & \Rightarrow & n_1=n_2=n_3=n_4=0, \\[1mm]
\{P_{+},J,P_{-}\}, & \Rightarrow & s_1=s_2=s_3=s_4=0, \\[1mm]
\{P_{+},P_{-},P_{+}\}, & \Rightarrow & u_1=u_2=u_3=u_4=0, \\[1mm]
\{P_{+},P_{-},T\}, & \Rightarrow & q_1=q_2=q_3=q_4=0,  \\[1mm]
\{P_{-},J,P_{-}\}, & \Rightarrow & p_1=p_2=p_3=p_4=0, \\[1mm]
\{J,P_{-},T\}, & \Rightarrow & v_1=ir_3, v_2=v_3=v_4=0,  \\[1mm]
\{J,P_{+},P_{-}\}, & \Rightarrow & l_1=l_2=l_4=0,  l_3=2r_3, \\[1mm]
\{P_{-},P_{+},P_{-}\}, & \Rightarrow & w_2=w_3=w_4=0, w_1=-2ir_3, \\[1mm]
\{T,P_{+},P_{-}\}, & \Rightarrow & r_3=0.
\end{array}\right.$$
Denoting $\alpha:=m_2$ we deduce the family of algebras $M(\alpha)$.
\end{proof}
In the next result we present precise description (up to isomorphism) of Leibniz algebras under the conditions of the subsection.

\begin{teo}\label{theorem6} An arbitrary Leibniz algebra of the family $M(\alpha)$ is isomorphic to one of the following non-isomorphic algebras $M(1)$ and $M(0).$
\end{teo}

\begin{proof} Since the elements $J, P_{+}, P_{-}, X_3, X_4$ are generators of an algebra of the family $M(\alpha)$, we take the general transformation of these elements:
$$\begin{array}{ll}
J'=A_1J+A_2P_{+}+A_3P_{-}+A_4T+A_5X_1+A_6X_2+A_7X_3+A_8X_4,\\[1mm]
P_{+}'=B_1J+B_2P_{+}+B_3P_{-}+B_4T+B_5X_1+B_6X_2+B_7X_3+B_8X_4,\\[1mm]
P_{-}'=C_1J+C_2P_{+}+C_3P_{-}+C_4T+C_5X_1+C_6X_2+C_7X_3+C_8X_4,\\[1mm]
X_3'=P_1X_1+P_2X_2+P_3X_3+P_4X_4, \\[1mm]
 X_4'=Q_1X_1+Q_2X_2+Q_3X_3+Q_4X_4.
\end{array}$$

The rest basis elements $T', X_1', X_2'$ can be find from the products:
$$T'=\frac{1}{2i}[P_{+}',P_{-}'], \quad X_1'=[X_3',P_{-}'], \quad X_2'=[X_1',P_{+}'].$$

Similarly as in the proof of Theorem \ref{theorem5} from the equalities
$$[T',P_{+}']=[T',P_{-}']=[J',T']=0, \quad X_2'=-2i[X_3',T'], \quad [J',J']=\alpha'X_2',$$
we derive the expression:
$$\alpha'=\frac{ A_1^2 \alpha}{B_2C_3P_3} \ \mbox{with} \ B_2C_3P_3\neq0.$$

If $\alpha\neq0,$ then by choosing $A_1:=\sqrt{\frac{B_2C_3P_3}{\alpha}}$ we obtain the algebra $M(1)$.

If $\alpha=0,$ then we get the algebras $M(0)$.
\end{proof}

\subsection{Leibniz algebras associated with the Diamond Lie algebra $\mathfrak{D}_{\mathbb{C}}$ and its Fock module.}

In this subsection we define Fock module over algebra $\mathfrak{D}_{\mathbb{C}}$. For algebra $\mathfrak{D}_{\mathbb{C}}$ we have the existence of a basis $\{J, P_1,P_2, T\}$  with the table of multiplications \ref{1.3}.

Let us introduce new notations for the basis elements of $\mathfrak{D}_{\mathbb{C}}:$
$$\overline{e}=J,\quad \overline{x}=P_1, \quad \frac { \overline{\delta}}{\delta x} = P_2, \quad \overline{1} = T.
$$

The action of the linear space ${\mathbb{C}}[x]$ on $\{\overline{1}, \overline{x}, \frac{\overline{\delta}}{\delta x}\}$ is induced by (\ref{Fock}). Further, we need to define the action on $\overline{e}$.

We set
$$(1,\overline{e})=\lambda+\lambda_1x+\lambda_2x^2+\dots+\lambda_nx^n.$$

Now we consider the Leibniz identity for the elements $1\in \mathbb{C}[x], \ \overline{e}, \frac{\overline{\delta}}{\delta x} \in \mathfrak{D}_{\mathbb{C}}:$
$$0=(1,[\overline{e},\frac{\overline{\delta}}{\delta x}])
=((1,\overline{e}),\frac{\overline{\delta}}{\delta
x})-((1,\frac{\overline{\delta}}{\delta x}),
\overline{e})=(\lambda+\lambda_1x+\lambda_2x^2+\dots+\lambda_nx^n,\frac{\overline{\delta}}{\delta
x}) =\sum\limits_{i=1}^ni\lambda_ix^{i-1}.$$

Therefore, $(1,\overline{e})=\lambda.$

From the equalities
$$x=(1,\overline{x})=(1,[\overline{e},\overline{x}])=((1,\overline{e}),\overline{x})-
((1,\overline{x}),\overline{e})=(\lambda,\overline{x})-(x,\overline{e})=\lambda
x-(x,\overline{e}),$$ we derive $(x,\overline{e})=(\lambda-1)x.$

By induction one can prove the equality:
$$(x^t,\overline{e})=(\lambda-t)x^t, \quad t\in{\mathbb{N}\cup\{0\}}.$$

Taking the change basis $\overline{e}':=\overline{e}-\lambda\overline{1}$ we can assume that $(x^t,\overline{e})=-tx^t, \ t\geq0.$

Therefore, the action on $\overline{e}$ is defined as follow:
$$(p(x),\overline{e}) \mapsto  -x\frac {\delta(p(x))}{\delta x}.$$

\begin{de} A linear space ${\mathbb{C}}[x]$ is called Fock $\mathfrak{D}_{\mathbb{C}}$-module, if there is an action $({\mathbb{C}}[x], \mathfrak{D}_{\mathbb{C}}) \mapsto {\mathbb{C}}[x]$ which satisfy the followings:
\begin{equation}\label{Fock1}
\begin{array}{lll}
(p(x),\overline{1})& \mapsto & p(x),\\
{}(p(x),\overline{x})& \mapsto & xp(x),\\
{}(p(x),\frac{\overline{\partial}}{\partial x})& \mapsto &\frac{\partial}{\partial x}(p(x)),\\
{}(p(x),\overline{e}) & \mapsto & -x\frac {\delta(p(x))}{\delta x}.
\end{array}
\end{equation}
for any $p(x) \in \mathbb{C}[x].$
\end{de}

The main result of this subsection consists of the classification of Leibniz algebras, whose corresponding Lie algebra is the complex Diamond Lie algebra $\mathfrak{D}_{\mathbb{C}}$ and the ideal $I$ is the Fock  $\mathfrak{D}_{\mathbb{C}}$-module.

\begin{teo}\label{1} The Leibniz algebra $L$ with conditions $L/I\cong \mathfrak{D}_{\mathbb{C}}$ and
 $I$ is the Fock $\mathfrak{D}_{\mathbb{C}}$-module, admits a basis
$$ \{\overline{1}, \overline{x}, \frac {\overline{\delta}}{\delta x}, \  \overline{e}, \ x^{t} \ | \ t\in \mathbb{N}\cup \{0\}\}$$
 such  that the table of multiplications in this basis has the following form:
$$\begin{array}{lll}
[\overline{e},\overline{x}]=\overline{x}, & [\overline{x}, \overline{e}]=-\overline{x}, \\[1mm]
[\overline{e},\frac {\overline{\delta}}{\delta x}]=-\frac {\overline{\delta}}{\delta x}, & [\frac {\overline{\delta}}{\delta x},\overline{e}]=\frac {\overline{\delta}}{\delta x},\\[1mm]
[\overline{x},\frac{\overline{\delta}}{\delta x}]=\overline{1},  & [\frac{\overline{\delta}}{\delta x},\overline{x}]=-\overline{1}, \\[1mm]
[x^t,\overline{1}] = x^{t}, & [x^t,\overline{x}] = x^{t+1}, \\[1mm]
[x^{t},\frac{\overline{\delta}}{\delta x}] = tx^{t-1}, & [x^t,\overline{e}] = -tx^{t},
 \end{array}$$ where the omitted products are equal to zero.
\end{teo}

\begin{proof}

Taking into account the action (\ref{Fock1}) we conclude that $ \{\overline{1}, \overline{x}, \frac {\overline{\delta}}{\delta x}, \  \overline{e}, \ x^{t} \ | \ t\in \mathbb{N}\cup \{0\}\}$ is a basis of $L$
and
$$[x^{t},\overline{1}] = x^{t}, \quad [x^t,\overline{x}] = x^{t+1},\quad [x^t,\frac {\overline{\delta}}{\delta x}] = tx^{t-1}, \quad  [x^t,\overline{e}] = -tx^{t}.$$
Let us denote
$$[\frac {\overline{\delta}}{\delta x},\overline{1}]=q(x),\quad
[\overline{1},\overline{1}]=r(x), \quad
[\overline{x},\overline{1}]=p(x), \quad
[\overline{e},\overline{1}]=m(x).$$

Taking the following change of basis elements:
$$\frac {\overline{\delta}}{\delta x}^{\prime} = \frac
{\overline{\delta}}{\delta x} - q(x), \quad \overline{1}^{\prime}
= \overline{1} - r(x), \quad \overline{x}^{\prime} = \overline{x}
- p(x),\quad \overline{e}^{\prime} = \overline{e} - m(x),$$ we obtain
$$[\overline{x},\overline{1}]=0, \quad [\frac {\overline{\delta}}{\delta x},\overline{1}]=0, \quad
[\overline{1},\overline{1}]=0, \quad
[\overline{e},\overline{1}]=0.$$

Consequently, $\overline{1}\in Ann_r(\mathfrak{D})$.

The chain of equalities
$[[\mathfrak{D},\mathfrak{D}],\overline{1}]=[\mathfrak{D},[\mathfrak{D},\overline{1}]]+[[\mathfrak{D},\overline{1}],\mathfrak{D}]=0$,
imply

$$[\overline{1},\overline{x}]=[\overline{1},\frac {\overline{\delta}}{\delta x}]=[\overline{1},\overline{e}]=[\overline{x}, \overline{x}]=[ \frac {\overline{\delta}}{\delta x},
\frac {\overline{\delta}}{\delta x}]=
[\overline{e}, \overline{e}]= 0,$$
$$
[\overline{x}, \frac {\overline{\delta}}{\delta x}]=-[\frac {\overline{\delta}}{\delta x}, \overline{x}]=\overline{1}, \quad [\overline{e},\overline{x}]=-[\overline{x},\overline{e}]=\overline{x},
\quad [ \frac {\overline{\delta}}{\delta x},\overline{e}]= -[\overline{e},\frac {\overline{\delta}}{\delta x}]=\frac {\overline{\delta}}{\delta x}.$$

\end{proof}

\subsection{Calculation of the spaces $BL^2$ and $ZL^2$}

In order to simplify routine calculations in the next subsection here we present two programs implemented in Mathematica 10, which calculate the general form of elements of the spaces $BL^2$ and $ZL^2$.

Below we explain the algorithm of the computer program which outputs the general form of 2-coboundary for a given Leibniz algebra. We input the dimension and the table of multiplications of a given $n$-dimensional Leibniz algebra, as well as a linear map $d(x_i)=\sum\limits_{k=1}^nb_{i,j}x_j$. The output of this program is the general form of 2-coboundary, that is, the map $f(x_i,x_j),$ which satisfy the condition:
$$f(x_i,x_j)=d([x_i,x_j])-[d(x_i),x_j]-[x_i,d(x_j)]. \eqno(16)$$

Initially, we introduce the necessary conditions for being Leibniz algebra and linear map. Next, we define a  Leibniz algebra $L$ through its table of multiplication and the linear map $d.$ Finally, we compute the general form of elements of $BL^2(L, L)$ imposing the condition (16).

The next program gives the general form of 2-cocycle for a given $n$-dimensional Leibniz algebra. Similarly, we input dimension of an algebra and its table of multiplication. Output of the program is $f(x_i,x_j)=\sum\limits_{k=1}^na_{i,j,k}x_k,$ which is the solution of the following equality:
$$[x_i, f(x_j, x_k)] + [f(x_i, x_k), x_j] - [f(x_i, x_j), x_k] - f([x_i, x_j], x_k) + f([x_i, x_k], x_j) + f(x_i, [x_j, x_k])=0. \eqno(17)$$

In fact, from these equality we find values $a_{i,j,k}$ in the expression $f(x_i,x_j)=\sum\limits_{k=1}^na_{i,j,k}x_k$. Thus, output of the program is the general form of a 2-cocycle $f(x_i,x_j)$.
The steps 1 and 2 of this program are analogous to the before. Finally, we compute the general form of elements of the space $ZL^2(L,L)$. 

Let us describe in details these steps:

\begin{itemize}
\item Description of equality (17)

\begin{verbatim}
ident[i_Integer, j_Integer, k_Integer] :=
Collect[mu[x[i], f[x[j], x[k]]] + mu[f[x[i], x[k]], x[j]] -mu[f[x[i], x[j]], x[k]]
- f[mu[x[i], x[j]], x[k]] + f[mu[x[i], x[k]], x[j]] + f[x[i], mu[x[j], x[k]]], base];
\end{verbatim}

\item Selection of non-zero coefficients in relation to the base. Equate selected coefficients to zero and solve the system of equations.

\begin{verbatim}
lista1 = Select[Flatten[Table[Coefficient[ident[i, j, k], base],
{i, 1, dim}, {j, 1, dim}, {k, 1, dim}]], ! NumberQ[#] &];

equations = Map[(# == 0 ) &, lista1];
resolution = Solve[equations, Reverse[par]][[1]];}
\end{verbatim}

\item Finally, substituting the solutions in the expression of $f$ we obtain the general form of 2-cocycles:

\begin{verbatim}
Module[{i,j},For[i = 1, i <= dim, i++,
For[j = 1, j <= dim, j++,
g[x[i], x[j]] =
Collect[f[x[i], x[j]] /. resolution, base, Simplify]]]];

application := Module[{i, j},
For[j = 1, j <= dim, j++,
For[i = 1, i <= dim, i++,
If[! NumberQ[g[x[i], x[j]]], Print["f[", x[i], ",", x[j], "]=",
g[x[i], x[j]]]]]]];
Print["General form of ZL^2:"] application
\end{verbatim}
\end{itemize}

Using the above mentioned two programs we find the general form of 2-coboundaries and 2-cocycles. After that we easily find a basis of the spaces $BL^2$ and $ZL^2,$ respectively. Later on, applying the methods of linear algebra we find the basis of the space $HL^2$.

\subsection{Linear deformations of Leibniz algebras associated with representations of the Diamond Lie algebra}

In this subsection using the computer programs of previous subsection we calculate the basis of the second group of
cohomologies for the finite-dimensional algebras obtained above. The verifications of integrability of linear deformations is also carried out by using computer programs.

\begin{pr} The basis of the space $HL^2(L_1,L_1)$ consists of cocycles $\{\varphi_1,\varphi_2,\varphi_3\},$ where
$$\varphi_1:\left\{\begin{array}{lll}
\varphi_ 1(X_ 1, J) = X_ 1,  &   \varphi_ 1(X_ 2, J) = X_ 2,   &  \varphi_ 1(X_ 3, J) = X_ 3, \\ [1 mm]
\end{array}\right.$$
$$\varphi_2:\left\{\begin{array}{lll}
\varphi_ 2(P_ -, P_ +) = J, &   \varphi_ 2[P_ +,
   P_ -] = -J, &  \varphi_ 2(X_ 2, P_ -) = -i/2 X_ 1, \\ [1 mm]
\varphi_ 2(X_ 1, T) = X_ 1/12, &   \varphi_ 2(X_ 2, T) =  X_ 2/12, &  \varphi_ 2(X_ 3, T) = -X_ 3/6, \\ [1 mm]
\end{array}\right.$$
$$\varphi_3:\left\{\begin{array}{lll}
\varphi_ 3(T, P_ +) =  P_ + , &  \varphi_ 3(T, P_ -) = -P_ -, &    \varphi_ 3(X_ 2, P_ -) =
 iX_ 1, \\ [1 mm]
\varphi_ 3(P_ +, T) = -P_ +, &  \varphi_ 3(P_ -, T) =  P_ -, &  \varphi_ 3(X_ 1, T) = X_ 1/2, \\ [1 mm]
\varphi_ 3(X_ 2, T) = -X_ 2/2. & &
\end{array}\right.$$
\end{pr}

Let us consider a linear deformation $\mu_t=\mu+t(a\varphi_1+b\varphi_2+c\varphi_3),$ where $\mu$ is multiplication law of the algebra $L_1$. Setting $a_1'=ta_1, \ a_2'=ta_2, \ a_3'=ta_3$ we can assume that parameter $t$ is equal to 1. Application of Leibniz identity for $\mu_t$ implies that a linear deformation $\mu_t$ is integrable if and only if $a_1a_2=0.$ Therefore, we have two linear deformations of the algebra $L_1:$
$$\mu_t^1=\mu+a_2\varphi_2+a_3\varphi_3, \quad \mu_t^2=\mu+a_1\varphi_1+a_3\varphi_3.$$

For the case of the algebra $L_2$ we have the following result.
\begin{pr} The basis of $HL^2(L_2,L_2)$ consists of cocycles $\{\varphi_1,\varphi_2,\varphi_3\},$ where
$$\varphi_1:\left\{\begin{array}{lll}
\varphi_1(X_1,J)=X_1,   & \varphi_1(X_2,J)=X_2, & \varphi_1(X_3,J)=X_3,
\end{array}\right.$$
$$\varphi_2:\left\{\begin{array}{lll}
\varphi_2(P_+,P_-)=-J,  & \varphi_2(P_-,P_+)=J, & \varphi_2(X_2,P_-)=-\frac{1}{2} i X_1, \\[1mm]
\varphi_2(X_1,T)=\frac{1}{12} X_1,  & \varphi_2(X_2,T)=-\frac{1}{12} X_2, & \varphi_2(X_3,T)=\frac{1}{6} X_3,
\end{array}\right.$$
$$\varphi_3:\left\{\begin{array}{lll}
\varphi_3(P_+,T)=-P_+,  & \varphi_3(T,P_+)=P_+, & \varphi_3(P_-,T)=P_-, \\[1mm]
\varphi_3(T,P_-)=-P_-,  & \varphi_3(X_2,P_-)=i X_1, & \varphi_3(X_1,T)=\frac{1}{2} X_1, \\[1mm]
\varphi_3(X_2,T)=-\frac{1}{2} X_2. & &.
\end{array}\right.$$
\end{pr}

Applying Leibniz identity to law $\mu_t$ deduces that a linear deformation $\mu_t=\mu+t(a_1\varphi_1+a_2\varphi_2+a_3\varphi_3)$ of algebra $L_2$ is integrable if and only if $a_2=0.$
Therefore, we obtain that any linear integrable deformation of algebra $L_2$ has the form $\mu_t=\mu+a_1\varphi_1+a_3\varphi_3$.

Since we are focused on linear integrable deformations of finite-dimensional Leibniz algebras obtained in this work, we shall omit the complete list of basis elements of $HL^2$ and we just present linear integrable deformations for the remaining algebras.

\begin{pr} An arbitrary linear integrable deformation of the algebra $L(1,0)$ has the following form:
$$\mu_t=\mu+a\varphi,$$
where
$$\varphi : \left\{\begin {array} {lll}
\varphi (P_ 1, P_ 1) =X_ 1, &  \varphi (T, P_ 1) =3 X_ 3,  &  \varphi (P_ 2, P_ 2) = X_ 1, \\[1 mm]
 \varphi(T, P_ 2) = -3 X_ 2, &  \varphi(J,T) = 2 X_ 1, & \varphi(P_ 1, T) = -2 X_ 3, \\[1 mm]
 \varphi(P_ 2, T) = 2 X_ 2. & &
 \end {array} \right.$$
\end{pr}

\begin{pr} Any linear deformation of algebra $L(1,1)$ (respectively, of the algebra $L(-1,1)$) is not integrable.
\end{pr}

\begin{pr} An arbitrary linear integrable deformation of algebra $L(0,1)$ has the following forms:
$$\mu_t=\mu+a\varphi \ \mbox{with} \ \varphi : \varphi (J, J) = X_ 4. $$
\end{pr}

\begin{pr} An arbitrary linear integrable deformation of algebra $L(0,0)$ has one of the following forms:

$$\mu_t^1=\mu+a_1\varphi_1+a_2\varphi_2, \quad  \mu_t^2=\mu+b_1\varphi_1+b_2\varphi_2-b_2\varphi_3,$$
 $$\mu_t^3=\mu+c_1\varphi_1+c_4\varphi_4, \quad \mu_t^4=\mu+d\varphi_5+d\varphi_6, \quad  \mu_t^5=\mu+k\varphi_7,$$
 where
$$\varphi_ 1: \varphi_ 1(J, J) = X_ 4, \quad  \varphi_2:\left\{\begin{array}{lll}
\varphi_ 2(P_ 2, J) =
 P_ 2, & \varphi_ 2(T, J) = T, & \varphi_ 2(X_ 3, J) = X_ 3, \\[1 mm]
\varphi_ 2(X_ 4, J) =
 X_ 4, & \varphi_ 2(J, P_ 2) = -P_ 2, & \varphi_ 2(J, T) = -T, \\[1 mm]
\end{array}\right.$$
$$\varphi_3:\left\{\begin{array}{lll}
\varphi_ 3(X_ 1, J) = X_ 1, & \varphi_ 3(X_ 3, J) =
 2 X_ 3, & \varphi_ 3([X_ 4, J) = X_ 4, \\[1 mm]
\varphi_ 3(X_ 1, P_ 2) = X_ 2, & \varphi_ 3(X_ 3, P_ 2) = -X_ 4. &
\end{array}\right.$$
$$\varphi_4:\left\{\begin{array}{lll}
\varphi_ 4(P_ 1, P_ 1) = X_ 1, & \varphi_ 4(T, P_ 1) =
 3 X_ 3, & \varphi_ 4(P_ 2, P_ 2) = X_ 1, \\[1 mm]
\varphi_ 4(T, P_ 2) = -3 X_ 2, & \varphi_ 4(J, T) =
 2 X_ 1, & \varphi_ 4(P_ 1, T) = -2 X_ 3, \\[1 mm]
\varphi_ 4(P_ 2, T) = 2 X_ 2. & &
\end{array}\right.$$
$$\varphi_5:\left\{\begin{array}{lll}
\varphi_ 5(P_ 2, P_ 1) = J, & \varphi_ 5(X_ 2, P_ 1) =
 X_ 1/4, & \varphi_ 5(X_ 4, P_ 1) = -X_ 3/4, \\[1 mm]
\varphi_ 5(P_ 1, P_ 2) = -J, & \varphi_ 5(X_ 3, P_ 2) =
 X_ 1/4, & \varphi_ 5(X_ 4, P_ 2) = X_ 2/4, \\[1 mm]
\varphi_ 5(X_ 2, T) = -X_ 3/2, & \varphi_ 5(X_ 3, T) = X_ 2/2. &
\end{array}\right.$$
$$\varphi_6:\left\{\begin{array}{lll}
\varphi_ 6(T, P_ 1) =
 P_ 2, & \varphi_ 6(X_ 2, P_ 1) = -X_ 1/4, & \varphi_ 6(X_ 4, P_ 1) =
 X_ 3/4, \\[1 mm]
\varphi_ 6(T, P_ 2) = -P_ 1, & \varphi_ 6(X_ 3, P_ 2) = -X_ 1/
   4, & \varphi_ 6(X_ 4, P_ 2) = -X_ 2/4, \\[1 mm]
\varphi_ 6(P_ 1, T) = -P_ 2, & \varphi_ 6(P_ 2, T) =
 P_ 1, & \varphi_ 6(X_ 2, T) = -X_ 3/2, \\[1 mm]
\varphi_ 6(X_ 3, T) = X_ 2/2. & &
\end{array}\right.$$
$$\varphi_7:\left\{\begin{array}{lll}
\varphi_ 7(P_ 1, X_ 1) = P_ 1, & \varphi_ 7(P_ 2, X_ 1) =
 P_ 2, & \varphi_ 7(T, X_ 1) = 2 T, \\[1 mm]
\varphi_ 7(X_ 2, X_ 1) = X_ 2, & \varphi_ 7(X_ 3, X_ 1) =
 X_ 3, & \varphi_ 7(X_ 4, X_ 1) = 2 X_ 4, \\[1 mm]
\varphi_ 7(J, X_ 2) = -P_ 2, & \varphi_ 7(P_ 2,
   X_ 2) = T, & \varphi_ 7(X_ 1, X_ 2) = -X_ 2, \\[1 mm]
\varphi_ 7(X_ 3, X_ 2) = X_ 4, & \varphi_ 7(J, X_ 3) =
 P_ 1, & \varphi_ 7(P_ 1, X_ 3) = -T, \\[1 mm]
\varphi_ 7(X_ 1, X_ 3) = -X_ 3, & \varphi_ 7(X_ 2,
   X_ 3) = -X_ 4, & \varphi_ 7(X_ 1, X_ 4) = -2 X_ 4. \\[1 mm]
\end{array}\right.$$
\end{pr}

\begin{pr} An arbitrary linear integrable deformation of algebra $M_1$ is
$$\mu_t=\mu+a_1\varphi_1+a_2\varphi_2+a_3\varphi_3+a_4\varphi_4$$
with
$$\varphi_1:\left\{\begin{array}{lll}
\varphi_ 1 (X_ 3, P_ +) = X_ 4, & \varphi_ 1 (X_ 3, T) = i /2 X_ 2, &
\end{array}\right.$$
$$\varphi_2:\left\{\begin{array}{lll}
\varphi_ 2 (P_ +, J) = P_ +, & \varphi_ 2 (T, J) = T, & \varphi_ 2 (X_ 2, J) = X_ 2 \\[1mm]
\varphi_ 2 (X_ 4, J) = X_ 4, & \varphi_ 2 (J, P_ +) = -P_ +, & \varphi_ 2 (J, T) = -T, \\[1mm]
\end{array}\right.$$
$$\varphi_3:\left\{\begin{array}{lll}
\varphi_ 3 (P_ -, J) = P_ -, & \varphi_ 3 (T, J) = T, & \varphi_ 3 (X_ 3, J) = -X_ 3, \\[1mm]
\varphi_ 3 (X_ 4, J) = -X_ 4, & \varphi_ 3 (J, P_ -) = -P_ -, & \varphi_ 3 (J, T) = -T, \\[1mm]
\end{array}\right.$$
$$\varphi_4:\left\{\begin{array}{lll}
\varphi_ 4 (X_ 1, J) = X_ 1, & \varphi_ 4 (X_ 2, J) = X_ 2, & \varphi_ 5 (X_ 3, J) = X_ 3, \\[1mm]
\varphi_ 4 (X_ 4, J) = X_ 4. & &
\end{array}\right.$$
\end{pr}

\begin{pr} An arbitrary linear integrable deformation of    algebra $M_2$ has one of the following forms: $$\mu_t^1=\mu+a_1\varphi_1+a_2\varphi_2+a_3\varphi_3, \quad \mu_t^2=\mu+b_3\varphi_3+b_4\varphi_4, \quad $$
with
$$\varphi_ 1 : \varphi_ 1 (J, J) = X_ 2, \quad \varphi_ 2: \varphi_ 2 (J, X_ 3) = X_ 2,$$
$$\varphi_ 3: \varphi_ 3 (X_ 3, P_+) = X_ 4, \quad   \varphi_ 3 (X_ 3, T) = i/2 X_ 2,$$
$$\varphi_4:\left\{\begin{array}{lll}
\varphi_ 4 (P_-, P_+) = J, &  \varphi_ 4 (P_+, P_-) = -J, &  \varphi_ 4 (X_ 2, P_-) = -i/2 X_ 1, \\[1 mm]
\varphi_ 4 (X_ 4, P_-) = -3 i/2 X_ 3,  &  \varphi_ 4 (X_ 1, T) =  X_ 1/4, &  \varphi_ 4 (X_ 2, T) = X_ 2/4, \\[1 mm] \varphi_ 4 (X_ 4, T) = -X_ 4/2. & &
\end{array}\right.$$

\end{pr}

\section*{Acknowledgments}

S. Uguz thanks HUBAK 13139 for the partial support. B.A. Omirov
thanks TUBITAK (2221 Program) for supports during his stay in
Turkey. He was also supported by Ministerio de Econom\'ia y
Competitividad (Spain), grant MTM2013-43687-P (European FEDER
support included) and by the Grant No.0828/GF4 of Ministry of Education and Science of the Republic of Kazakhstan. The second author was supported by Xunta de Galicia, grant GRC2013-045 (European FEDER
support included).

\end{document}